\documentclass[a4paper, 10pt]{amsart}

\usepackage{amsmath}
\usepackage{amsfonts}
\usepackage{amssymb}
\usepackage{amsthm}
\usepackage{mathtools, verbatim}
\usepackage[utf8]{inputenc}
\usepackage{indentfirst}
\usepackage{enumerate}
\usepackage[colorlinks=true]{hyperref}
\hypersetup{urlcolor=blue, citecolor=red}
\usepackage{graphicx}
\usepackage[usenames,dvipsnames]{pstricks}
\usepackage{epsfig}
\usepackage{pst-grad} % For gradients
\usepackage{pst-plot} % For axes
\usepackage{ifthen}
\usepackage{color, xcolor}

\def\sideremark#1{\ifvmode\leavevmode\fi\vadjust{\vbox to0pt{\vss% the remark
 \hbox to 0pt{\hskip\hsize\hskip1em%                          will appear only
 \vbox{\hsize2.1cm\tiny\raggedright\pretolerance10000%          on the side
  \noindent #1\hfill}\hss}\vbox to15pt{\vfil}\vss}}}%

\allowdisplaybreaks

\usepackage{fullpage}
\numberwithin{equation}{section}

\def\polhk#1{\setbox0=\hbox{#1}{\ooalign{\hidewidth\lower1.5ex\hbox{`}\hidewidth\crcr\unhbox0}}}

\def\XXint#1#2#3{{\setbox0=\hbox{$#1{#2#3}{\int}$ }
\vcenter{\hbox{$#2#3$ }}\kern-.6\wd0}}

\newcommand{\osc}{\operatorname{osc}}

\newcommand{\tr}{\operatorname{Tr}}

\theoremstyle{plain}
\newtheorem{Theorem}{Theorem}[section]

\newtheorem{Definition}[Theorem]{Definition}
\newtheorem{Lemma}[Theorem]{Lemma}

\newtheorem{Proposition}[Theorem]{Proposition}
\newtheorem{Remark}[Theorem]{Remark}
\newtheorem{Assumption}{A}

\begin{document}
\title{ Interior regularity estimates for fully nonlinear equations with arbitrary nonhomogeneous degeneracy laws}

\author{P\^edra D. S. Andrade and Thialita M. Nascimento}
\maketitle
	
\begin{abstract}

In this paper, we study regularity estimates for a class of degenerate, fully nonlinear elliptic equations with arbitrary nonhomogeneous degeneracy laws.  We establish that viscosity solutions are locally continuously differentiable under suitable conditions on the degeneracy laws. Our proof employs improvement of flatness techniques alongside an alternative recursive algorithm for renormalizing the approximating solutions, linking our model to the homogeneous, fully nonlinear, uniformly elliptic equation.

\end{abstract}
\medskip


\medskip 


\section{Introduction}
We investigate the interior regularity of viscosity solutions for a class of degenerate, fully nonlinear equations, given by
\begin{equation}
\mathfrak{F}(x, Du, D^2 u) = f(x) \quad \text{in} \quad B_1,
\end{equation}
where $f$ is a bounded function in the unit ball $B_1 \subset \mathbb{R}^d$, $d\ge 2$, and  $\mathfrak{F}: B_1\times \mathbb{R}^d \times \mathcal{S}(d) \rightarrow \mathbb{R}$ represents a nonlinear degenerate elliptic operator. The operator $\mathfrak{F}$ is governed by two distinct degeneracy laws, which influence its behavior in a critical manner. These laws permit the ellipticity of 
$\mathfrak{F}$ to degenerate as the solution approaches the critical region $\{Du=0\}$, where the gradient of the solution vanishes. A prototype we have in mind is $\mathfrak{F}(x, z, M)= \left[\sigma(|z|) + a(x)\nu(|z|) \right]F(M)$. Hereafter, we are interested in studying equations of the form
\begin{equation}\label{eq_main}
\smallskip
\left[\sigma(|Du|) + a(x)\nu(|Du|) \right]F(D^2 u) = f(x) \quad \text{in} \quad B_1,
\end{equation}
where $\sigma, \nu$ are moduli of continuity, $0\le a(\cdot) \in \mathcal{C}(B_1)$, $F$ is a fully nonlinear uniformly elliptic operator and $f$ is a bounded function in $B_1$. 

Degenerate fully nonlinear equations with degeneracy of double-phase type arise in several contexts, such as optimization, stochastic control, and geometric problems. Especially, models that incorporate a nonhomogeneous degeneracy law are particularly significant in Materials Science.  Models with nonhomogeneous degeneracies offer a more realistic representation of diverse phenomena observed in sciences. For instance, applications of equation \eqref{eq_main} can be found in image restoration problems, as discussed in \cite{HH}.

One of the motivations for studying this class of degenerate problems arises from the Calculus of Variations. Specifically, the model introduced by Zhikov in \cite{Zhikov1993, Zhikov1995, Zhikov1997} to investigate homogenization theory and the occurrence of the Lavrentiev phenomenon, namely
\begin{equation}\label{funct_double-phase}
\int_{B_1} \Phi(x, Dz) \, \mbox{d}x,
\end{equation}
where $\Phi$ denotes the double-phase integrand, given by
\[
\Phi(x, Dz):= |Dz|^p + a(x)|Dz|^q, \quad \quad 1<p< q,
 \]
 see \cite{ZKO-1994}. These models are crucial for studying the behavior of strongly anisotropic materials, where properties associated with the integrand $\Phi$, dependent on the gradient and $x$, can vary significantly from point to point (see, for instance, \cite{CM2015} and references therein). The functional \eqref{funct_double-phase} falls within the class of integral functionals with nonstandard growth conditions, as defined by Marcellini in \cite{Marcellini1989, Marcellini1991}. 
 
 Due to its importance, regularity theory for the minimizers of \eqref{funct_double-phase} has been extensively studied in the past few decades. Colombo and Mingione established optimal regularity results for the minimizers of \eqref{funct_double-phase} by assuming $a(\cdot) \in \mathcal{C}^{0, \alpha}$ with $0<\alpha \le 1$ and $ \frac{q}{p}< 1 + \frac{\alpha}{d}$ in \cite{CM2015}.  Specifically, they showed that minimizers are locally of class $ \mathcal{C}^{1, \alpha}$. In addition, Calder\'on-Zygmund estimates have been obtained in the works \cite{Colombo-Mingione2016, DeFilippis-Mingione2019}. The existence and multiplicity of solutions to double-phase problems of the type \eqref{funct_double-phase} have been proved in \cite{Liu-Dai2018}. We refer the reader to \cite{PMG18, PMG15, Colombo-Mingione2015, ELM2004, Hasto-Ok2022} and references therein for a comprehensive overview of this topic.

The nonvariational counterparts of the models mentioned above can be formulated as singular or degenerate fully nonlinear problems. Equations involving operators whose ellipticity degenerate or blow up as a power of the gradient have been extensively studied in the past few decades. For an account of this matter, see \cite{Araujo-Ricarte-Teixeira2015, Birindelli-Demengel2004, Birindelli-Demengel2007, Birindelli-Demengel2009, Birindelli-Demengel2006, Davila-Felmer-Quaas2010, Davila-Felmer-Quaas2009, Imbert2011, Imbert-Silvestre2013, Imbert-Silvestre2016}. 

In the degenerate setting, the work-horse  model for such equations is 
\begin{equation}\label{First model}
|Du|^p F(D^2 u ) = f(x) \quad \text{in} \quad B_1,
\end{equation}
where $p>0, f\in L^{\infty}(B_1)$ and $F$ is a nonlinear uniformly elliptic operator. 

A notion of first eigenvalue for a more general class of degenerate fully nonlinear models as \eqref{First model} has been introduced by Birindelli and Demengel in a series of papers of \cite{Birindelli-Demengel2004, Birindelli-Demengel2007, Birindelli-Demengel2006}. The authors showed the existence, uniqueness, and maximum principle for these models. Moreover, Birindelli and Demengel proved the comparison principle and Liouville-type results in \cite{Birindelli-Demengel2004}. Alexandroff-Bakelman-Pucci estimate, Harnack inequality and H\"older estimates for the solutions of \eqref{First model} are proved by D\'avila, Felmer, and Quaas in \cite{Davila-Felmer-Quaas2010, Davila-Felmer-Quaas2009} and Imbert in \cite{Imbert2011}.

Regularity estimates for the gradient of the solutions of \eqref{First model} were proved by Imbert and Silvestre in \cite{Imbert-Silvestre2013}. The authors employed the H\"older estimates obtained in \cite{Imbert2011} combined with a Lipschitz estimate derived from the Crandall-Ishii-Lions technique. In \cite{Araujo-Ricarte-Teixeira2015}, Ara\'ujo, Ricarte, and Teixeira have proved the optimal regularity estimates for the solutions of 
\begin{equation}
H(x, Du) F(x, D^2 u) = f(x) \quad \text{in} \quad B_1,
\end{equation}
where $f \in L_{\infty}(B_1)$ and $H: B_1 \times \mathbb{R}^d \rightarrow \mathbb{R}$ satisfying $\lambda |z|^p \le H(x, z) \le \Lambda |z|^p$ for some $p>0$ and $F$ assumed to be a fully nonlinear uniformly elliptic operator. In particular, by using the Crandall-Ishii-Lions method combined with approximation techniques, the authors established that the solutions to \eqref{First model} are locally of class ${\mathcal C}^{1,\alpha_*}$, where $\alpha_* \in (0, 1)$ satisfies
$$
    \alpha_* <  \min \left( \alpha_0, \frac{1}{1+p} \right)
$$ and $\alpha_0 \in (0, 1)$  is the universal exponent in the Krylov-Safonov theory available for $F = 0$. See \cite[Theorem 3.1]{Araujo-Ricarte-Teixeira2015} for more details.

Many authors have addressed further developments, including the case of variable exponents. For example, see \cite{Bronzi-Pimentel-Rampasso-Teixeira} and the references therein. For general degeneracy laws, the first author and their collaborators established $\mathcal{C}^1$- regularity estimates for the viscosity solutions of
\begin{equation}\label{eq_main2} 
\sigma(|Du|) F(D^2 u) = f \quad \text{in} \quad B_1,
\end{equation}
 provided $\sigma: \mathbb{R}_+ \rightarrow \mathbb{R}_+$ is a modulus of continuity whose the inverse $\sigma^{-1}$ is Dini continuous,  $F: {\mathcal{S}(d)} \rightarrow \mathbb{R}$ is a fully nonlinear $(\lambda, \Lambda)$-uniformly elliptic operator and $f$ is continuous and bounded function in $B_1$. For more details, we refer the reader to \cite{APPT}.  
 
 In \cite{PS}, Pimentel and Stolnicki investigated the existence of viscosity solutions for the associated Dirichlet problem and the corresponding free transmission problem for a similar type of degeneracy as in the equation \eqref{eq_main2}. The authors also utilize a Dini continuity condition, as in \cite{APPT}, to establish the differentiability of the solutions.
 
Moreover, under minimal assumptions on the operator, Baasandorj, Byun, Lee, and Lee in \cite{BBLL24} proved optimal regularity results for viscosity solutions of 
 \[
 \Theta(x, |Du|)F(x, D^2u) = f(x) \quad \text{in} \quad B_1
 \]
 showing that solutions are locally of class $\mathcal{C}^{1, \alpha}$ for all $\alpha >0$ satisfying smallness conditions, where $F$ is a fully nonlinear uniformly elliptic operator, $\Theta$ satisfies minimal conditions and $f\in \mathcal{C}(B_1)\cap L^{\infty}(B_1).$

The study of models with double-phase degeneracies in the nondivergence setting can be expressed as follows:
\begin{equation}\label{holder model}
\smallskip
\left[|Du|^p + a(x)|Du|^q \right]F(D^2 u) = f(x) \quad \text{in} \quad B_1,  \quad \quad 0< p \le q.
\end{equation}
was initiated, as far as we know,  in the work of De Filippis in  \cite{DeFilippis2021}. The author demonstrated that the gradient of the solutions satisfies $\mathcal{C}^{0,\alpha}$-estimates, where the exponent $0 < \alpha \le \frac{1}{1+p}$. In \cite{daSilva-Ricarte2020}, Da Silva and Ricarte established sharp regularity estimates for \eqref{holder model}, that is, the authors proved $\mathcal{C}^{1,\alpha_*}$,
for 
$$
    \alpha_* <  \min \left( \alpha_0, \frac{1}{1+p} \right)
$$ 
and again $\alpha_0 \in (0, 1)$  is the universal exponent in the Krylov-Safonov theory available for $F = 0$. For an account of the double degeneracies with variable exponent models, see the survey \cite{BSRRV2022}.

In this manuscript, we are interested in studying viscosity solutions to more general double-phase degenerate equations such as \eqref{eq_main}. First, we observe that the double-phase type equation emerges when $a$ changes sign, specifically as the modulating coefficient $a$ becomes zero or positive. Notably, on the zero level set $\{ a=0\}$, the equation \eqref{eq_main} reduces to
\begin{equation}\label{eq_APPT}
\sigma(|Du|) F(D^2 u) = f(x) \quad \text{in} \quad \{ a=0\} \cap B_1,
\end{equation}
where $\sigma$ is a modulus of continuity that degenerates as a function of the gradient, $F$ is a nonlinear uniformly elliptic operator and $f \in L^{\infty}(B_1)$. Under the additional condition that $\sigma^{-1}$ is Dini continuous, we recover the known results established in \cite{APPT}. In particular, viscosity solutions of equation \eqref{eq_APPT} are locally of class $\mathcal{C}^1$. On the other hand, on the level set $\{a>0\}$, equation \eqref{eq_main} simplifies to
\begin{equation}\label{eq_main01}
\smallskip
\left[\sigma(|Du|) + a(x)\nu(|Du|) \right]F(D^2 u) = f(x) \quad \text{in} \quad \{ a>0\}\cap B_1,
\end{equation}
where $\sigma$, $\nu$ and $F$ and $f$ satisfy the same conditions as in  \eqref{eq_main}. In this case, regularity results remain unknown. Throughout this manuscript, we aim to establish regularity estimates for the model given by \eqref{eq_main01} within the domain $B_1$.  

In other words, equation \eqref{eq_main} is also characterized as a diffusive model, where the ellipticity is affected by a vanishing rate of the gradient of the solutions and the modulating function $a$, both of them are critical in the analysis of this class of partial differential equations. A significant point of interest in this model lies in the uncertainty surrounding the geometry and qualitative properties of the domains $\{ a=0\}\cap B_1,$ and $\{ a>0\}\cap B_1$. Although an interesting line of research, this matter is not covered in the present study.

Observe that the regularity results in \cite{daSilva-Ricarte2020, DeFilippis2021}, see also  \cite{BdSRR} and \cite{Fang-Radulescu-Zhang2021} (for the variable-exponent case), are under the natural order for laws of degeneracy. Heuristically, if $\sigma_1 \prec \sigma_2$, in the sense that 
$$
    \sigma_1 (t) = o (\sigma_2(t)),
$$
then one expects that solutions to equations with a $\sigma_2$ degeneracy law are smoother than those of the corresponding class of equations with $\sigma_1$ law. In other words, a stronger degeneracy law results in a less diffusive model, which consequently diminishes the system's smoothing properties. Thus, for a double-phase degeneracy law as in \eqref{eq_main}, the stronger degeneracy law is expected to quantitatively govern the regularity estimates of the solutions.

 The main goal of this present paper is to establish local differentiability for solutions within this class of equations. To achieve this, we impose suitable conditions on the moduli of continuity $\sigma$ and $\nu$ and employ an improvement of flatness technique, combined with an alternative recursive algorithm that renormalizes the approximating solutions, connecting our model to the homogeneous, fully nonlinear uniformly elliptic equation $F(D^2 u )= 0$.
 
 Our main result in this direction states as follows:

\begin{Theorem}[Differentiability of the solutions]\label{main_theorem}
Let $u \in \mathcal{C}(B_1)$ be a normalized viscosity solution to 
\begin{equation}\label{eq_main1}
\left[\sigma(|Du|) + a(x)\nu(|Du|) \right]F(D^2 u) = f(x) \quad \text{in} \quad B_1,
\end{equation}
where $0\le a(\cdot) \in \mathcal{C}(B_1)$ and $\sigma(\cdot), \nu(\cdot)$ are moduli of continuity with inverses $\sigma^{-1}$, $\nu^{-1}$. Suppose $F$ is a fully nonlinear $(\lambda, \Lambda)$-uniformly elliptic operator,  $\nu (t) = o(\sigma(t))$, $\nu^{-1}$ is Dini continuous,  and $f \in L^{\infty}(B_1)$. Then $u $ are locally of class $\mathcal{C}^1$ and there exists a modulus of continuity $\omega: [0, + \infty) \rightarrow [0, + \infty)$ depending only upon dimension, $\lambda, \Lambda$, $\sigma, \nu$,   and $\| f\|_{L^{\infty}(B_1)}$ such that
\[
|Du(x) - Du(y)| \le \omega(|x - y|)
\]
for every $x, y \in B_{1/4}$.
\end{Theorem}

It is worth pointing out, that we deal with the case where the moduli of continuity are not necessarily H\"older continuous. Therefore,  Theorem \ref{main_theorem} extends the results of \cite{DeFilippis2021} to equations with arbitrary law of degeneracy as studied in \cite{APPT}.

\begin{comment}
Hence,  as shown in \cite{APPT}, to preserve the differentiability of the solution one must prevent the degeneracy law $\sigma(|z|) + a(x)\nu(|z|)$ from approaching zero too abruptly.
\end{comment}

 To demonstrate Theorem \ref{main_theorem}, we employ the well-known improvement of flatness technique. A major challenge is ensuring that the degeneracy law $\sigma + a(\cdot) \nu$ does not itself degenerate. To address this, we require a non-collapsing property for the term $\sigma + a(\cdot) \nu$. Finally, we employ a recursive geometric construction that establishes the existence of hyperplanes locally comparable to the solutions of scaled equations, for which this non-collapsing property holds. The summability of $\nu^{-1}$, guaranteed by its Dini continuity property, ensures the convergence of this sequence of hyperplanes, thereby completing the proof.

The remainder of this paper is organized as follows: in Section \ref{Prelim sct}, we provide formal definitions, main assumptions and auxiliary results required in this paper.  In Section \ref{Compctness sct}, we prove the universal continuity of the viscosity solution. Sections \ref{Improv_Flatness} and \ref{Regularity sct} are devoted to the proof of Theorem \ref{main_theorem}. 

\section{Preliminaries}\label{Prelim sct}
In this section, we collect basic notions and detail the main assumptions and auxiliary results used throughout this paper. 
\subsection{Main assumptions and  definitions}
We start this section by fixing some notations:  Let $\mathbb{R}^d$ denote the $d$-dimensional Euclidean space. $\mathcal{S}(d)$ denotes the space of real $d \times d$ symmetric matrices. For a number $r >0$, $B_r(x_0)$ will denote the open ball of center $x_0$ and radius $r >0$. When $x_0=0$ we will simply write $B_r$ instead of $B_r(0)$. Next, we present the uniform ellipticity condition of the operator $F$.

\begin{Assumption}[Uniform ellipticity]\label{A1}
We assume the nonlinear operator $F: \mathcal{S}(d) \rightarrow$ $\mathbb{R}$ is uniformly elliptic for some $0<\lambda\leq \Lambda$. That is,
\begin{equation}\label{eq_lambdaeliptic}
	\lambda\|N\| \leq F(M+N)-F(M) \leq \Lambda\|N\|
\end{equation}
for every $M, N \in \mathcal{S}(d)$, with $N \geq 0 $. In addition, for normalization purposes, we assume $F(0)=0$. 
\end{Assumption}
Hereafter, we refer to any operator $F$ satisfying \eqref{eq_lambdaeliptic} as a $(\lambda, \Lambda)$-ellipitc operator.  

The supremum and infimum of all $(\lambda, \Lambda)$-elliptic operators are called the {\it Pucci operators}. In fact, we have the following:
\begin{Definition}[Pucci's extremal operators]
  Let $\lambda$ and $\Lambda$ be positive constants such that $\lambda \le \Lambda$ and $M \in \mathcal{S}(d)$. We set
  \begin{equation}\label{Pucci_-}
\mathcal{M}_{\lambda, \Lambda}^-(M):= \inf_{A \in \mathcal{A}_{\lambda, \Lambda}} \tr(AM)
\end{equation}
 and 
 \begin{equation} \label{Pucci_+}
\mathcal{M}_{\lambda, \Lambda}^-(M):= \sup_{A \in \mathcal{A}_{\lambda, \Lambda}} \tr(AM),
\end{equation}
where $\tr(M)$ is the trace of $M$, and $\mathcal{A}_{\lambda, \Lambda}$ is the space of all symmetric $d \times d$ matrix whose eigenvalues belong to $[\lambda, \Lambda]$.
In particular, since $M = OD O^t$, where $D = e_i \delta_{ij}$ ($e_i$ are the eugenvalues of $M$) and $O$ is an orthogonal matrix,  it is easy to see that  
  \[
  \mathcal{M}_{\lambda, \Lambda}^-(M):= \lambda \sum_{e_i>0} e_i + \Lambda \sum_{e_i<0} e_i
  \]
  and 
   \[
  \mathcal{M}_{\lambda, \Lambda}^+(M) := \Lambda \sum_{e_i>0} e_i + \lambda \sum_{e_i<0} e_i,
  \]
  where $e_i = e_i(M)$ are the eigenvalues of $M$.
\end{Definition}

With the definition of $\mathcal{M}^{\pm}$ at hand, it is equivalent that an operator $F$ is uniformly elliptic if 
\begin{equation}\label{def_ellipticitywPucci}
\mathcal{M}_{\lambda, \Lambda}^-(Y) \leq F(X+Y) -F(Y)\leq  \mathcal{M}_{\lambda, \Lambda}^+(Y), 
\end{equation}
for any $X, Y \in \mathcal{S}(d)$. 

Next, we assume that the modulating coefficient $a(\cdot)$ is nonnegative and continuous in $B_1$ and that the source function $f$ is bounded. That is, 
\begin{Assumption} \label{assump_modulating_coefficient}
$a(\cdot) \ge 0$ and $a(\cdot) \in \mathcal{C}(B_1)$, 
\end{Assumption}
and
\begin{Assumption}\label{assump_f}\rm
 $f\in L^\infty(B_1).$ 
\end{Assumption}

In the sequel,  we assume the appropriate condition on the degeneracy law required for the differentiability result proved in this manuscript. 
\begin{Assumption}\label{assump_moduli}
We suppose $\sigma, \nu : [0, + \infty) \rightarrow [0, + \infty)$ are moduli of continuity satisfying
\begin{equation}\label{moduli decay}\tag{A 4.1}
    \nu(t) \le \sigma(t) ,\, \, \text{for all} \, \,  t \in [0,1].
\end{equation}
We further assume that $\nu$ has an inverse and that 
\begin{equation}\label{Dini cond for nu}\tag{A 4.2}
    \nu^{-1} \quad \text{is  Dini continuous}
\end{equation} and  finally we assume 
\begin{equation}\label{normalization of nu} \tag{A 4.3}
    \nu (1) \ge 1.
\end{equation} The ultimate assumption is a mere normalization. 
\end{Assumption}
In particular, the condition \eqref{moduli decay} implies that 
\begin{equation} \label{normalization_moduli}
\sigma(1) \ge \nu (1) \ge 1.
\end{equation}
We recall that a modulus of continuity is a function $\omega: [0, +\infty) \to [0,+\infty)$ which is monotone increasing and $\lim\limits_{t \to 0} \omega(t) = 0$.  Hence, we observe that $\tilde{\sigma}:= \sigma + a(x)\nu$ is also a modulus of continuity because  $\sigma$ and $\nu$ are modulus of continuity and $a$ is a nonnegative function that does not depend on the gradient. Given the significance of Dini continuity property to the differentiability result, we provide its formal definition.

\begin{Definition}[Dini condition]\label{def_dini}
We say that a modulus of continuity $\omega$  is Dini continuous if it satisfies 
\begin{equation}\label{DC}
	\int_0^\tau\frac{\omega (t)}{t}{\bf d}t\,<\,+\infty,
\end{equation}
for some $\tau>0$. 
\end{Definition}
\begin{Remark}
   An equivalent definition of Dini continuity states that a function $\omega$ is Dini continuous if it satisfies the following condition
\[
\sum_{k =1}^{\infty} \omega(\tau \theta^k) < \infty,
\]
for any $0<\theta<1$.
\end{Remark}

\begin{Remark}
    Notice that the Dini continuity of $\nu^{-1}$  implies the Dini continuity of $\sigma^{-1}$. This follows from inequality   \eqref{moduli decay}, the normalization condition \eqref{normalization_moduli}, and the monotonicity of $\nu$ and $\sigma$. 
\end{Remark}

The Dini continuity is pivotal to our analysis. As shown in \cite{APPT}, the Dini continuity of the inverse degeneracy law enables the construction of a {\it non-collapsing} sequence $(\sigma_j)_{j \in \mathbb{N}}$ (to be defined later) of the degeneracies laws. The concept of non-collapsing sets was introduced in \cite{APPT} to investigate the interior regularity of viscosity solutions of \eqref{eq_APPT}, see Definition \ref{def_noncollapsing} below. In what follows, we recall and define this notion to our context, as it plays a crucial role in establishing the Approximation Lemma in Section \ref{Improv_Flatness}.

\begin{Definition}[{\cite[Definition 3]{APPT}}]\label{def_noncollapsing}
Let $\Sigma$ be a collection of moduli of continuity defined over an interval $I \in \mathcal{I}$. We call $\Sigma$ is a non-collapsing set if for all sequences $(f_j)_{j \in \mathbb{N}} \subset \Sigma$, and all sequences $(a_j)_{j \in \mathbb{N}} \subset I$, then 
\[
f_j(a_j)\rightarrow 0 \quad \text{implies} \quad a_j \rightarrow 0.
\]
Here $\mathcal{I}:=\{ (0, T] \, |\, 0 < T < \infty\} \cup \{\mathbb{R}_0^+\}$.
\end{Definition}

\begin{Remark}
It is worth highlighting that \cite[Proposition 2]{APPT} states that $\Sigma$ is a non-collapsing set, if and only if, for all sequences $(f_j)_{j \in \mathbb{N}} \subset \Sigma$ and $a\in I\setminus\{0\}$, $\liminf_{j \rightarrow \infty} f_j (a) >0$. 
\end{Remark} 

Another method for generating a family of non-collapsing moduli of continuity is through a shoring-up process. For instance, see \cite[Proposition 5]{APPT}. 

\begin{Definition}[ {\cite[Definition 5]{APPT}}, Shoring-up] A sequence of moduli of continuity $(\sigma_j)_{j\in\mathbb{N}}$ is said to be shored-up if there exists a sequence of positive numbers $(c_j)_{j \in \mathbb{N}}$ converging to zero  such that
$$
    \inf\limits_{j}  \sigma_j (c_j)  > 0 \quad \text{for all} \, \, j \in \mathbb{N} .
$$
\end{Definition}

We aim to link the expected regularity estimates for $\left[\sigma(|Du|) + a(x)\nu(|Du|) \right]F(D^2 u) = f$ to the one available for $F(D^2 u) = 0$. To achieve this, we require a type of {\it cancellation} effect, interpreted in the viscosity sense as in \cite{APPT}. For this purpose, we define a {\it non-collapsing} property to our setting. This notion is essential to ensure that the degeneracy law remains well-defined and does not degenerate itself.
 
\begin{Definition}[Non-collapsing property]\label{assump_noncollapsing}
    We say that a family of functions  $\mathcal{F} (B_1, \mathbb{R}) = \{ H : B_1 \times [0, + \infty] \to [0, +\infty] \}$ has the non-collapsing property if for any (fixed) $x \in B_1$ and for all sequences $( H_j (x, \cdot))_{j \in \mathbb{N}} \subset \mathcal{F} $ and $(c_j)_{j \in \mathbb{N}}\subset I$, we have
    $$
        H_j(x, c_j) \to 0 \quad \text{implies} \quad c_j \to 0, 
    $$
where $ I$ is an interval in $ \mathcal{I}:=\{ (0, T] \, |\, 0 < T < \infty\} \cup \{\mathbb{R}_0^+\}$.
\end{Definition}

\begin{Proposition}\label{noncollapsing for inhomog}
    Let $(\sigma_j)_{j \in \mathbb{N}}$ and $(\nu_j)_{j \in \mathbb{N}}$ be moduli of continuity sequences satisfying A\ref{assump_moduli}, and a modulating coefficient sequence $(a_j)_{j \in \mathbb{N}}$ satisfying A\ref{assump_modulating_coefficient}. For all $j \in \mathbb{N}$, we assume $(\sigma_j)_{j \in \mathbb{N}}$ and $(\nu_j)_{j \in \mathbb{N}}$ belong to a family of non-collapsing moduli of continuity $\Sigma$ and define $\mathcal{H} = \{\tilde{\sigma}_j:  \sigma_j(t) + a_j(x) \nu_j(t) \} $ to be a family of modulus of continuity.  Then, $\mathcal{H}$ has the non-collapsing property. 
 \end{Proposition}

 \begin{proof}
    It is straightforward from the Assumptions A\ref{assump_moduli} and A\ref{assump_modulating_coefficient} that $\tilde{\sigma}_j$ is a modulus of continuity for all $j \in \mathbb{N}$. Given a sequence $(\tilde{\sigma}_j)_{j \in \mathbb{N}}:=(\sigma_j + a_j(x) \nu_j)_{j \in \mathbb{N}} \subset \mathcal{H}$ and for all sequences $(c_j)_{j \in \mathbb{N}} \subset I$ such that 
     $$
          \sigma_j(c_j)  + a_j(x) \nu_j (c_j) \to 0. 
     $$
     We have that, since  $ \sigma_j(c_j)  + a_j(x) \nu_j (c_j) \ge \sigma_j(c_j)$ for all $j \ge 1$, there holds
     $$
        0 = \limsup\limits_{j \to \infty}  \left( \sigma_j(c_j)  + a_j(x) \nu_j (c_j)\right)  \ge \limsup\limits_{j \to \infty}  \sigma_j(c_j).
     $$
     Hence, $\lim\limits_{j \to \infty}  \sigma_j(c_j) = 0$. Since $(\sigma_j)_{j \in \mathbb{N}} \in \Sigma$, implies that $c_j \to 0$ for all $j \in \mathbb{N}$. This completes the proof.
 \end{proof}

Next, we recall the notion of viscosity solution used in the paper. 
\begin{Definition}[Viscosity solution]
    Let $G \colon B_1 \times \mathbb{R}^d \times\mathcal{S}(d) \to \mathbb{R}$ be a degenerate elliptic operator. We say that $u$ is a viscosity subsolution to
$$
    G(x,Du,D^2u)=0 \quad \mbox{in} \quad B_1,
$$
if for every $x_0 \in B_1$ and  $\varphi \in \mathcal{C}^2\left(B_1\right)$, such that $ u -\varphi$ attains a local maximum at $x_0$, we have 
$$
    G(x_0, D\varphi(x_0), D^2\varphi(x_0)) \geq 0.
$$
Conversely, we say that $u$ is a viscosity supersolution to
$$
   G(x,Du,D^2u)=0 \quad \mbox{in} \quad B_1,
$$
if for every $x_0 \in B_1$ and  $\varphi \in \mathcal{C}^2\left(B_1 \right)$, such that $u -\varphi$ attains a local minimum at $x_0$ we have 
$$
   G(x_0, D\varphi(x_0), D^2\varphi(x_0)) \leq 0.
$$
A function $u$ is said to be a viscosity solution if it is both a sub and supersolution. 

\end{Definition}

We will call $u \in L^{\infty}(B_1)$ a normalized viscosity solution, if $\| u \|_{L^{\infty}(B_1)} \le 1$. For a full account of the theory of viscosity solutions,  we refer the reader to  \cite{Caffarelli-Cabre1995} and \cite{Crandall-Ishii-Lions1992}. 

We conclude this subsection by discussing the scaling features of the model and the smallness regime with which we will be working.
\begin{Remark}[Smallness regime]
Throughout the paper, we require smallness conditions for the solution $u$ and for the right-hand side $f$ in the equation \eqref{eq_main}, namely
\begin{equation}\label{smallness_condition}
\|u \|_{L^{\infty}(B_1)} \le 1 \quad \quad \text{and} \quad \quad \|f \|_{L^{\infty}(B_1)} \le \varepsilon,
\end{equation}
for some $\varepsilon > 0$ to be determined. Notice that the conditions in \eqref{smallness_condition} are not restrictive. Indeed, consider the function
\[
v(x) := \frac{u(rx)}{K},
\]
for $0< r\ll 1$ and $K>0$ to de chosen.  We have that $v$ solves 
\begin{equation}\label{eq_smallness-condition}
\left[\tilde{\sigma}(|Dv|) + \tilde{a}(x) \tilde{\nu}(|Dv|)\right] \tilde{F}(D^2 v)= \tilde{f}(x) \quad \quad \text{in} \quad \quad B_1,
\end{equation}
where
\[
\tilde{\sigma}(t):= \sigma\left( \frac{K}{r}t\right),\quad \quad\tilde{a}(x):= a(rx), \quad \quad\tilde{\nu}(t):= \nu\left( \frac{K}{r}t\right) 
\]
and 
\[
\tilde{F}(M):= \frac{r^2}{K} F\left(\frac{K}{r^2}M\right), \quad \quad \tilde{f}(x):= \frac{r^2}{K} f(rx) 
\]
Now, we observe that if $\sigma$ and $\nu$ adimit inverse then 
\[
\tilde{\sigma}^{-1}(t):= \frac{r}{K} \sigma^{-1}\left( t\right)\quad \quad \text{and} \quad \quad\tilde{\nu}^{-1}(t):= \frac{r}{K} \nu^{-1}\left( t\right).
\]
In fact, 
\[
\tilde{\sigma}^{-1}(\tilde{\sigma}(t))= \tilde{\sigma}^{-1}\left( \sigma\left( \frac{K}{r}t\right)\right)=  \frac{r}{K} \sigma^{-1}\left( \sigma\left( \frac{K}{r}t\right)\right) = t.
\]
Arguing similarly, we obtain $ \tilde{\nu}^{-1}(\tilde{\nu}(t))= t$. Hence, by taking $r<K$, we have $\tilde{\sigma}^{-1}$ and $\tilde{\nu}^{-1}$ satisfy the Dini continuous property, that is, 
\[
\int_0^1 \frac{\tilde{\sigma}^{-1}(t)}{t} {\bf d}t\le \int_0^1 \frac{\sigma^{-1}(t)}{t} {\bf d} t \quad \quad \text{and} \quad \quad \int_0^1 \frac{\tilde{\nu}^{-1}(t)}{t} {\bf d} t\le \int_0^1 \frac{\nu^{-1}(t)}{t}{\bf d} t.
\]
In addition, 
\[
\tilde{\sigma}(1) = \sigma\left(\frac{K}{r}\right) \ge \sigma(1) \ge 1\quad \quad \text{and} \quad \quad \tilde{\nu}(1) = \nu\left(\frac{K}{r}\right) \ge \nu(1) \ge 1
\]
Therefore, $\sigma$ and $\nu$ satisfy $A\ref{assump_moduli}$. Also, it follows from the definition of $\tilde{F}$ that is a $(\lambda, \Lambda)$-elliptic operator. Finally, by choosing

\[
r:=\varepsilon \quad \quad \text{and} \quad \quad K:=\frac{1}{\| u \|_{L^{\infty}} + \| f \|_{L^{\infty}}},
\]
we obtain \eqref{smallness_condition}, and thus equation \eqref{eq_smallness-condition} belongs to the same class of the operator in \eqref{eq_main}.

 \end{Remark}

 \begin{Remark}[Invariance of the solutions by constants] \label{Rmk-solution}
 Since the equation \eqref{eq_main} does not depend on $u$,  then for any solution $u$  of \eqref{eq_main}, we have that $u_c = u + c$ is also a solution to \eqref{eq_main} for any constant $c \in \mathbb{R}$. Hence, we assume $u(0)=0$ without loss of generality.
 \end{Remark}
 
\subsection{Foundational results}
 In this subsection, we present a collection of auxiliary results that will be used throughout the paper.  We start by stating the {\it Crandall-Ishii-Lions} Lemma that plays a key role in our analysis.

\begin{Proposition}[ {\cite[Theorem 3.2]{Crandall-Ishii-Lions1992}}]\label{IshiiLions}
Let $G:B_1\times\mathbb{R}^d\times\mathcal{S}(d)\to\mathbb{R}$ be a degenerate elliptic operator. Let $\Omega \subset B_1$, $u \in  \mathcal{C}(B_1)$ and $\psi$ be twice continuously differentiable in a neighborhood of $\Omega\times \Omega$. Set 
\[
w(x, y): = u(x) - u(y) \quad \text{for} \quad (x, y) \in \Omega\times \Omega .
\]
If the function $w - \psi$ attains the maximum at $(x_0, y_0) \in \Omega\times \Omega$, then for each $ \varepsilon > 0$, there exist $ X, Y \in {\mathcal S}(d)$ such that
\begin{equation*}\label{viscosity_solution}
G(x_0, D_x \psi(x_0, y_0), X) \leq 0 \leq G(y_0, D_y \psi(x_0, y_0), Y),
\end{equation*}
and the block diagonal matrix with entries X and Y satisfies
\begin{equation*} \label{matrix_inequality}
- \left( \frac{1}{\varepsilon} + \|A \|\right) I \leq 
\left( 
\begin{array}{cc}
X & 0 \\
0 & -Y \\
\end{array}
\right)
\leq  A + \varepsilon A^2, 
\end{equation*}
where $A = D^2 \psi(x_0, y_0)$.
\end{Proposition}

The following lemma establishes the existence of a shored-up sequence of laws of degeneracy, which can be employed in the approximation technique in Section \ref{approx lemma}, by ensuring that the convergence of the approximating hyperplanes remains maintained.

\begin{Lemma}[{\cite[Lemma 2]{APPT}}]\label{existence of noncollapsing}
   Let $A$ be a compact subset of $\ell_1$ with $0 \notin A$. Given numbers $\varepsilon , \delta > 0$, there exists a sequence of positive numbers $(c_j)_{j \in\mathbb{N}} \in c_0$, satisfying  $\max_{j \in \mathbb{N}}|c_j| \le {\varepsilon}^{-1}$ such that
   $$
   \left(\frac{a_j}{c_j}\right)_{j \in \mathbb{N}} \in \ell_1
   $$
   and
   \begin{equation}
       \varepsilon \left( 1 - \frac{\delta}{2} \right)\| (a_j)\|_{\ell_1} \le \left\| \left( \frac{a_j}{c_j}\right) \right\|_{\ell_1} \le \varepsilon(1 + \delta) \|(a_j)\|_{\ell_1} ,
   \end{equation}
  for all $(a_j)_{j \in \mathbb{N}} \in A$.
\end{Lemma}

\section{Compactness for a family of PDEs with arbitrary nonhomogeneous degeneracy}\label{Compctness sct}

In this section, we establish interior H\"older regularity estimates for a class of fully nonlinear $(\lambda, \Lambda)$-elliptic equations with nonhomogeneous degeneracies. It is important to emphasize that this result is proven without imposing any additional conditions on the moduli of continuity $\sigma$ and $\nu$. The argument follows the lines of \cite[Proposition 1.]{APPT} taking into account that $\nu$ nonincreasing function and that $a(\cdot)\ge 0$. Furthermore, it plays a fundamental role in the proof of Theorem \ref{main_theorem}.

\begin{Lemma}[H\" older continuity] \label{Holder-continuity-lemma}
Let $u \in \mathcal{C}(B_1)$ be a normalized viscosity solution to
\begin{equation}\label{eq_compactness}
\left[\sigma(|Du + \xi|) + a(x)\nu(|Du + \xi|) \right]F(D^2 u) = f(x) \quad \text{in} \quad B_1,
\end{equation}
where $\xi$ is an arbitrary vector in $\mathbb{R}^d$. Suppose A\ref{A1}, A\ref{assump_f}, A\ref{assump_modulating_coefficient} and \eqref{normalization_moduli} hold true. Then $u \in \mathcal{C}^{0, \alpha}_{loc}(B_1)$, for some $\alpha  \in (0, 1).$
\end{Lemma}

\begin{proof}
 First, we observe that the proof follows the same lines as in \cite{APPT} with few modifications, which we write down for completeness.
 
 Let us fix $0 < \eta < 1$ and define the modulus of continuity $\omega: [0, + \infty) \rightarrow [0, + \infty) $ as $\omega(t) = t^{\eta}$. For some $0 < r \ll 1$ to be chosen later, we shall prove that there exist positive numbers $L_1$ and $L_2$ such that
\[
\mathcal{L}:= \sup_{x, y \in B_r} \left( u(x) -u(y) - L_1 \omega(|x-y|) - L_2\left(|x - x_0|^2 + |y - x_0|^2\right)\right) \le 0,
\]
for all $x_0 \in B_{r/2}$. Suppose by contradiction that for all positive constants $L_1$ and $L_2$, there exists $x_0 \in B_{r/2}$ for which $\mathcal{L}>0$.

Next, we introduce the auxiliaries functions $\psi, \phi: \overline{B}_r \times \overline{B}_r \rightarrow \mathbb{R}$ given by 
\[
\psi(x, y)= L_1\omega (|x - y|) + L_2(|x - x_0|^2 + |y - x_0|^2) 
\]
and
\[
\phi(x, y) = u(x) - u(y) - \psi(x, y).
\]
Since $\phi$ is a continuous function in $\overline{B}_r \times \overline{B}_r$, there exists a maximum point $(\bar{x}, \bar{y})$ in $\overline{B}_r \times \overline{B}_r$, that is, $\phi(\bar{x}, \bar{y}) = \mathcal{L}$. Hence, 
\[
\psi(\bar{x}, \bar{y}) < u(\bar{x}) - u(\bar{y}) \leq \osc_{B_1} u\leq 2,
\]
which implies
\[
L_1\omega (|x - y|) + L_2(|x - x_0|^2 + |y - x_0|^2)  \leq 2 
\]

By setting $L_2:= \left(\frac{4 \sqrt{2}}{r}\right)^2$, we obtain
\[
|\bar{x} - x_0|^2 \leq \frac{r}{4} \quad \text{and} \quad |\bar{y} - x_0|^2  \leq \frac{r}{4}.
\]
Thus, we can conclude that $\bar{x}, \bar{y}$ are interior points of $B_r$. Observe that it is straightforward $\bar{x} \not= \bar{y}$; otherwise, $\mathcal{L} \leq 0$ trivially. 

Now, we have $D_x\psi$ and $D_y \psi$ at $(\bar{x}, \bar{y})$ stand for 
\[
D_x\psi (\bar{x}, \bar{y})= L_1 \omega^{\prime} (|\bar{x} - \bar{y}|)|\bar{x} - \bar{y}| ^{-1} (\bar{x} - \bar{y}) + 2 L_2(\bar{x} - x_0)
\]
and

\[
D_y\psi(\bar{x}, \bar{y}) = L_1 \omega^{\prime} (|\bar{x} - \bar{y}|)|\bar{x} - \bar{y}| ^{-1} (\bar{x} - \bar{y}) + 2 L_2(\bar{y} - x_0)
\]
To simplify, let us fix the notation
\[
\xi_{\bar{x}} = D_x\psi (\bar{x}, \bar{y}) \quad \text{and} \quad \xi_{ \bar{y}} = D_y\psi (\bar{x}, \bar{y}).
\]

Applying Proposition \ref{IshiiLions}, we obtain that for every $\varepsilon>0,$ there exist matrices $X, Y \in \mathcal{S}(d)$ satisfying the viscosity inequalities 
\begin{equation}\label{inequality_operators}
\left[\sigma(|\xi_{\bar{x}} + \xi|) + a(\bar{x}) \nu(|\xi_{\bar{x}} + \xi|) \right]F(X) - f(\bar{x})\leq 0 
\leq 
\left[\sigma(|\xi_{\bar{y}} + \xi|) + a(\bar{y}) \nu(|\xi_{\bar{y}} + \xi|) \right]F(Y) - f(\bar{y})
\end{equation}
and 
\begin{equation} \label{main-inequality}
\left( 
\begin{array}{cc}
X & 0 \\
0 & -Y \\
\end{array}
\right)
\leq  
\left( 
\begin{array}{cc}
Z & -Z \\
-Z & Z \\
\end{array}
\right)2 L_2 I + \varepsilon A^2, 
\end{equation}
where $A = D^2 \psi(\bar{x}, \bar{y})$ and
\[
Z:= L_1 \omega^{\prime \prime}(|\bar{x} - \bar{y}|) \frac{(\bar{x} - \bar{y})\oplus(\bar{x} - \bar{y})}{|\bar{x} - \bar{y}|^2} + L_1 \frac{\omega^{\prime}(|\bar{x} - \bar{y}|)}{|\bar{x} - \bar{y}|} \left( I -  \frac{(\bar{x} - \bar{y})\oplus(\bar{x} - \bar{y})}{|\bar{x} - \bar{y}|^2} \right).
\]

In what follows, we evaluate some suitable vectors in the matrix inequality \eqref{main-inequality} to obtain information about the eigenvalues of $X - Y$. Consider $v \in \mathbb{S}^{d - 1}$ and the vector $(v, v) \in \mathbb{R}^{2d}$. Hence, 
\[ 
\langle (X-Y)v, v  \rangle \le (4L_2 + 2 \varepsilon \zeta),
\]
with $\zeta:= \| A^2\|$. Therefore, we conclude that all eigenvalues of $X - Y$ are below $4L_2 + 2 \varepsilon \zeta$. Now, consider vectors of the form $(\bar{v}, - \bar{v}) \in \mathbb{R}^{2d}$ such that
\[
\bar{v}:= \frac{\bar{x} - \bar{y}}{|\bar{x} - \bar{y}|}.
\]
Applying to \eqref{main-inequality}, we have
\[
\langle (X - Y)\bar{v}, \bar{v}\rangle \leq 4 L_1 \omega^{\prime \prime}(|\bar{x} - \bar{y}|) + (4L_2 + 2 \varepsilon \zeta) |\bar{v}|^2
\]
Since $\omega$ is twice differentiable, $\omega>0$ and $\omega^{\prime \prime} <0$, implies that at least an eigenvalue of $X- Y$ is below $ - 4 L_1  + (4L_2 + 2 \varepsilon \zeta) $. Notice that for $L_1$ large enough, we have $ - 4 L_1  + (4L_2 + 2 \varepsilon \zeta) $ becomes negative. 

Calculating the Pucci's operator associated with $X- Y$, we obtain
\begin{equation}\label{Pucci_inequality}
\mathcal{M}_{\lambda, \Lambda}^- (X - Y) \geq 4 \lambda L_1 - (\lambda + (d - 1)\Lambda)(4L_2 + 2 \varepsilon \zeta).
\end{equation}
Since $F$ is a $(\lambda, \Lambda)$-uniformly elliptic operator, we can use the inequalities \eqref{inequality_operators} 
 and \eqref{Pucci_inequality} to get the following estimate
\begin{equation} \label{last_estimate}
4 \lambda L_1  \le (\lambda + (d - 1)\Lambda)(4L_2 + 2 \varepsilon \zeta) + \frac{f(\bar{x})}{\sigma(|\xi_{\bar{x}} + \xi|) + a(\bar{x}) \nu(|\xi_{\bar{x}} + \xi|)} - \frac{f(\bar{y})}{\sigma(|\xi_{\bar{y}} + \xi|) + a(\bar{y}) \nu(|\xi_{\bar{y}} + \xi|)}
\end{equation}

Now, consider a positive constant $\tilde{C}$ to be determined later. We analyze two different cases.

{\it Case 1}: $|\xi|> \tilde{C}$.

First, we estimate the value of $|\xi_{\bar{x}}|$. Notice that from the definition of $\xi_{\bar{x}}$, we have that $|\xi_{\bar{x}}|$ is less than $ L_1 |\omega^{\prime}(|\bar{x} - \bar{y}|)| + 2 L_2$, which implies 
\begin{equation*}
|\xi_{\bar{x}}| \leq \tilde{c} L_1, 
\end{equation*}
where $\tilde{c}> 0$ is a constant universal. Now, by choosing $\tilde{C}:= 10\tilde{c}L_1> 1$, with $L_1$ to be determined later. Hence,
\[
|\xi + \xi_{\bar{x}}| \ge \tilde{C} - \frac{\tilde{C}}{10} = \frac{9}{10} \tilde{C}.
\]
Similarly, we obtain $|\xi_{\bar{y}}| \leq \tilde{c} L_1.$ Thus,
\[
|\xi + \xi_{\bar{y}}| \ge \tilde{C} - \frac{\tilde{C}}{10} = \frac{9}{10} \tilde{C}.
\]
Using the fact that $a(\cdot)>0$ combined with $\nu$ is a nonnegative function, we obtain the first inequality below. To conclude the following estimate, we use that $\sigma$ is a nondecreasing function and satisfies \eqref{normalization_moduli}. Therefore
\begin{equation}\label{ineq_01}
\frac{f(\bar{x})}{\sigma(|\xi_{\bar{x}} + \xi|) + a(\bar{x}) \nu(|\xi_{\bar{x}} + \xi|)} \le \frac{\| f \|_{L^{\infty}(B_1)}}{\sigma(|\xi_{\bar{x}} + \xi|) } \le \frac{| f \|_{L^{\infty}(B_1)}}{\sigma(\frac{9}{10} \tilde{C})} \le | f \|_{L^{\infty}(B_1)}
\end{equation}
and 
\begin{equation}\label{ineq_02}
\frac{- f(\bar{y})}{\sigma(|\xi_{\bar{x}} + \xi|) + a(\bar{x}) \nu(|\xi_{\bar{x}} + \xi|)} \le \frac{\| f \|_{L^{\infty}(B_1)}}{\sigma(|\xi_{\bar{y}} + \xi|) } \le \frac{| f \|_{L^{\infty}(B_1)}}{\sigma(\frac{9}{10} \tilde{C})} \le | f \|_{L^{\infty}(B_1)}
\end{equation}
Plugging the inequalities \eqref{ineq_01} and \eqref{ineq_02} in \eqref{last_estimate}, we have 
\begin{equation*} 
4 \lambda L_1  \le (\lambda + (d - 1)\Lambda)(4L_2 + 2 \varepsilon \zeta) + 2 | f \|_{L^{\infty}(B_1)}.
\end{equation*}
Finally, we take $L_1= L_1(\lambda, \Lambda, d, L_2, r) \gg 1$ large enough, which give us a contradiction. Therefore $\mathcal{L} \le 0$, which implies that the solutions of \eqref{eq_compactness} are locally $\eta$-H\"older continuous.

{\it Case 2}: $|\xi| \le \tilde{C}$, where $\tilde{C} $ is defined as in the previous case. 

Consider the operator
\[
\tilde{F}(x, p, M):= \left[\sigma(|p + \xi|) + a(x) \nu(|p + \xi|) \right]F(M). 
\]

The equation can be written as $\tilde{F}(x, Du, D^2 u) = f(x)$. Thus in particular, if $|p| \ge 5 \tilde{C}$, then $| p + \xi | \ge 4\Tilde{C} \ge 1$ and  hence, $\left[\sigma(|p + \xi|) + a(x) \nu(|p + \xi|) \right] \ge \sigma(4\Tilde{C} ) \ge 1$. Therefore,  $\tilde{F}(x, p, M) = f(x)$ with $|p| \ge 5 \tilde{C}$ implies
$$
    \left\{\begin{matrix}
\mathcal{M}^{+} (D^2 u) + |f| \ge 0  \\
 \mathcal{M}^{-} (D^2 u) - |f| \le 0  \\
\end{matrix}\right. 
$$
where $\mathcal{M}^{\pm}$ are the extremal Pucci operators associated with the ellipticity constants of $F$, therefore, it is known from \cite{Imbert-Silvestre2016} and \cite{Delarue2010}, that $u$ is H\"older continuous with estimates depending only on the dimension, ellipticity constants, $\|f\|_{L^{\infty}(B_1)}$ and $\Tilde{{C}}$.

Together with the previous case, this completes the proof of the theorem. 
\end{proof}

\section{Improvement of flatness}\label{Improv_Flatness}
In this section, we present a key approximation result essential to our analysis. Specifically, under appropriate smallness conditions, we show that a normalized viscosity solution to \eqref{eq_main} can be approximated by a function in ${\mathcal{C}}_{loc}^{1, \beta}(B_1)$ for some $\beta \in (0, 1)$.

\begin{Lemma}[Approximation Lemma]\label{approx lemma}
Let $\mathfrak{S}$ be a collection of non-collapsing moduli of continuity and $u \in \mathcal{C}(B_1)$ be a normalized viscosity solution to 
\begin{equation*}
\left[\sigma(|Du + \xi|) + a(x)\nu(|Du + \xi|) \right]F(D^2 u) = f(x) \quad \text{in} \quad B_1,
\end{equation*}
where $\xi  \in \mathbb{R}^d$ is an arbitrary vector and $\sigma(\cdot), \nu(\cdot) \in \mathfrak{S}$. Assume that A\ref{A1} - A\ref{assump_moduli} are in force. Given $\delta > 0$ there exists $\varepsilon > 0$ such that if $\|f\|_{L^{\infty}(B_1)} < \varepsilon$, then one can find a function $h \in {\mathcal{C}}_{loc}^{1, \beta}(B_1)$ such that 
$$
    \| u - h \|_{L^{\infty}(B_{1/2})} < \delta
$$
where $\beta \in (0, 1)$ is a universal constant. 

\end{Lemma}
\begin{proof}
To simplify the presentation, we split the proof into five steps. 
\medskip

\noindent{\it Step 1.} We argue by contradiction. Suppose that the statement does not hold. Then, there exist sequences $(\sigma_j)_{j \in \mathbb{N}}$,  $(\nu_j)_{j \in \mathbb{N}}$,  $({\xi}_j)_{j \in \mathbb{N}}$,  $(a_j)_{j \in \mathbb{N}}$,  $(u_j)_{j \in \mathbb{N}}$,  $(F_j)_{j \in \mathbb{N}}$,  $(f_j)_{j \in \mathbb{N}}$ and a positive number $\delta_0$ such that, for every $j\in \mathbb{N}$, we have $u_j$ is a normalized viscosity solution to
\begin{equation}\label{eq_approximation-lemma}
\left[\sigma_j(|Du_j + \xi_j|) + a_j(x)\nu_j(|Du_j + \xi_j|) \right]F_j(D^2 u_j) = f_j(x) \quad \text{in} \quad B_1,
\end{equation}
where $u_j(0)= 0$, $F_j:\mathcal{S}(d) \rightarrow \mathbb{R}$ is a $(\lambda, \Lambda)$-elliptic operator and $\sigma_j$ and $\nu_j$ are moduli of continuity satisfying 
\begin{itemize}
\item[(i)] $ a_j \in \mathcal{C}(B_1) $ with $a_j (\cdot ) \ge 0$;

\item[(ii)] $\sigma_j(0)=0$ and $\nu_j(0)=0$;

\item[(iii)] $ \sigma_j(1)\ge \nu_j(1)\ge 1$;

\item[(iv)]If $\sigma_j(t_j) + a_j(x)\nu_j(t_j) \rightarrow 0$  then $t_j \rightarrow 0$ . 
\end{itemize} 
and 
\[
\|f_j\|_{L^{\infty}(B_1)}\le \frac{1}{j},
\]
however
\[
\sup_{x\in B_{1/2}} |u_j(x) - h(x)| \ge \delta_0,
\]
 for all $h \in \mathcal{C}_{loc}^{1, \beta}(B_1) $, where $\beta >0 $ is a universal constant to be determined further in the proof. 
 \medskip
 
\noindent {\it Step 2. } It follows from the uniform ellipticity property along the sequence $(F_j)_{j \in \mathbb{N}}$ that there exists $F_{\infty}$ such that $F_{j}$  converges uniformly to $F_{\infty}$ as $j \rightarrow\infty$. In addition, by applying Lemma \ref{Holder-continuity-lemma} we get that
\[
u_j \in {\mathcal{C}_{loc}^{0, \alpha}} (B_1) \quad \text{and}\quad \|u_j\|_{\mathcal{C}_{loc}^{0, \alpha}} (B_1)  \le C,
\]
where the constants $\alpha \in (0,1)$ and $C>0$ do not depend on $j$. Hence, by applying the Arzelà-Ascoli Theorem, there exists a subsequence $(u_j)_{j \in \mathbb{N}}$ and a continuous function $u_{\infty}$ such that  $u_j$ converges uniformly to $u_{\infty}.$ Now, we aim to show that $ u_{\infty}$ is a normalized viscosity solution to
\begin{equation*} 
F_{\infty}(D^2v) = 0 \quad \text{in} \quad B_{9/10}.
\end{equation*}
\medskip
Without loss of generality, we can assume the test function as a quadratic polynomial $p$, namely
$$
p(x): = u_{\infty} (y) + b \cdot (x - y) + \frac{1}{2} (x -y )^{T} M (x-y).
$$
Also, we can assume that $p$ touches $u_{\infty}$ from bellow at $y$ in $B_{3/4}$ \emph{i.e.,} $p(x) \le u_{\infty} (x)$, $x \in B_{3/4}$, and clearly $p(y) = u_{\infty}(y)$. Without loss of generality, we may assume  $y = 0$. Our goal is to verify that 
\begin{equation}\label{goal equation}
    F_{\infty} (M) \le  0.
\end{equation}
\medskip 

\noindent {\it Step 3.} For $ 0 < r \ll 1$ fixed, define the sequence $(x_j)_{j \in \mathbb{N}}$ such that $x_j \rightarrow 0$ and satisfies
$$
    p(x_j) - u_j(x_j) = \max\limits_{x \in B_r} \left( p(x) - u_j(x_j) \right).
$$
From PDE satisfied by $u_j$  in \eqref{eq_approximation-lemma},  we infer that 

\begin{equation}\label{ineq1-AL}
    \left[ \sigma_j (|b + \xi_j |) + a_j(x_j) \nu_j(| b + \xi_j | ) \right] F_j(M) \le f_j (x_j) . 
\end{equation}

\smallskip

From here, we analyze several cases:

\smallskip
\noindent{\it Case 1.} Let $( \xi_j )_{ j \in \mathbb{N}}$ be an unbounded sequence, then we consider the (renamed) subsequence satisfying $| \xi_j | > j$, for every $j \in \mathbb{N}$. There exists $j^* \in \mathbb{N}$ such that 
$$
    |b + \xi_j | > 1, \, \forall \, j > j^*.
$$
By conditions (i) and (iii), we have  $ 1 \le \sigma_j (|b + \xi_j |) + a_j(x_j) \nu_j (|b + \xi_j|)$. Hence, multiplying \eqref{ineq1-AL} by $\left[ \sigma_j (|b + \xi_j |) + a_j(x_j) \nu_j(| b + \xi_j | ) \right]^{-1}$, we obtain
\smallskip 

$$
F_j(M) \le f_j (x_j) ,
$$
\smallskip

\noindent for every $j > j^*$ . By letting $j \to \infty$,  we get $F_{\infty} (M) \le  0$, as desired. 

\smallskip

\noindent{\it Case 2.} Now, let $( \xi_j )_{ j \in \mathbb{N}}$ be a bounded sequence, then at least through a subsequence
\smallskip
$$
     b + \xi_{j} \longrightarrow b + \xi_{\infty}.
$$
\smallskip
\noindent{\it Case 2.1.} If $|b + \xi_{\infty}| > 0$ then, by property (iv), we have $\sigma_j (|b + \xi_j |) + a_j(x_j) \nu_j(| b + \xi_j | ) \nrightarrow 0$. Therefore, 

$$
    F_j(M) \le \frac{f_j (x_j)}{\sigma_j (|b + \xi_j |) + a_j(x_j) \nu_j(| b + \xi_j | )} \to 0.
$$

\noindent Finally, we reach $F_{\infty} (M) \le 0$, as desired. It remains to analyze the case $|b + \xi_{\infty}| = 0$ ({\it i.e.,} $\xi_{\infty} = - b$ ).

\smallskip
\noindent{\it Case 2.2.} 
In the sequel, we study the case $|b + \xi_j| \to 0$ as $j \to \infty.$ First, notice that if $\mbox{Spec}( M ) \subset ( - \infty, 0 ]$, then ellipticity yields \eqref{goal equation}. Indeed, by ellipticity
$$
    F_{\infty} (M) \le \mathcal{M}_{\lambda, \Lambda}^{+} (M) = \lambda \sum\limits_{i = 1}^{d} \tau_i \le 0
$$
where $\{\tau_i\}_{i=1}^{d}$ are the eigenvalues of $M.$
Thus, we may assume that $M$ has $k \ge1 $ strictly positive eigenvalues.  Let $\{ e_i \}_{i=1}^{k}$ be the associated eigenvectors and define 
$$
    E := \text{Span}\{ e_1 , \dots, e_k \}. 
$$
Next, we consider the direct sum $\mathbb{R}^{d} : = E \oplus G$ and the orthogonal projection $P_{E}$ into $E$. For $\kappa>0$ small enough, we define the test function
\begin{equation} \label{Test_function-AL}
\varphi(x) := \kappa \sup_{ e \in \mathbb{S}^{d-1}} \langle P_E x, e \rangle + b \cdot x + \frac{1}{2} x^T M x.
\end{equation}
Because $u_j \to u_{\infty}$ locally uniformly in $B_1$, and then $\varphi$ touches $u_{j}$ from below at some point $x_j^{\kappa} \in B_r$ such that $x_j^{\kappa} \rightarrow 0$ as $j$ goes to infinity, for every $0 < \kappa \ll 1$ and $j \gg 1$. 

We study the two cases: $x_j^{\kappa} \in G$ and $x_j^{\kappa} \notin G$. First, if $x_j^{\kappa} \in G$, then we may rewrite \eqref{Test_function-AL} as
\[
\kappa \langle P_E x, e \rangle + b \cdot x + \frac{1}{2} x^T M x
\]
touches $u_j$ at $x_j^{\kappa}$ regardless of the direction of $e \in \mathbb{S}^{d-1}$. It is easy to see that
\[
D\langle e,  P_E x\rangle  = P_E (e) \quad \text{and} \quad D^2\langle e,  P_E x\rangle = 0.
\]
Also, we have 
\[
e \in \mathbb{S}^{d-1} \cap E \Longrightarrow  P_E (e) =e \quad \text{and} \quad e \in \mathbb{S}^{d-1} \cap G \Longrightarrow  P_E (e) =0.
\]
Testing the function above in the inequality \eqref{ineq1-AL}, we obtain  
\begin{equation}\label{Case2.2 - 01}
    \left[ \sigma_j (|M x_j^{\kappa}  + b + \xi_j + \kappa e |) + a_j(x_j^{\kappa}) \nu_j(|M x_j^{\kappa} +  b + \xi_j + \kappa e | ) \right] F_j(M) \le f_j (x_j^{\kappa}) 
\end{equation}
for every $e \in \mathbb{S}^{d-1}$. Since, the sequences $|M x_j^{\kappa} |$ and $ |b + \xi_j|$ go to zero, we can fix $\tilde{j} \in \mathbb{N}$ such that $|M x_j^{\kappa} | + |b + \xi_j| \le \frac{\kappa}{2}$ for every $j > \tilde{j}$. Hence, 
\[
|M x_j^{\kappa}  + b + \xi_j + \kappa e | \ge \frac{\kappa}{2}.
\]
Therefore, 
\begin{eqnarray}\label{Case2.2 - estimate1}
\frac{1}{ \left[ \sigma_j (|M x_j^{\kappa}  + b + \xi_j + \kappa e |) + a_j(x_j^{\kappa}) \nu_j(|M x_j^{\kappa} +  b + \xi_j + \kappa e | ) \right]}
&\le& \frac{1}{\sigma_j (\frac{\kappa}{2}) + a_j(x_j^{\kappa}) \nu_j (\frac{\kappa}{2})} \nonumber \\
&\le &  \frac{1}{\sigma_j(\frac{\kappa}{2})}.
\end{eqnarray}
Combining \eqref{Case2.2 - 01} and \eqref{Case2.2 - estimate1}, we get 
$$
    F_j(M) \le \frac{f_j(x_j^{\kappa})}{\sigma_j (\frac{\kappa}{2}) + a_j(x_j^{\kappa}) \nu_j (\frac{\kappa}{2})} \le \frac{\|f_j(x_j^{\kappa})\|_{L^{\infty}(B_1)}}{\sigma_j (\frac{\kappa}{2})} 
 \to 0,
$$
as $j \to \infty,$ \emph{i.e.,} $F_{\infty}(M) \le 0$.

To complete the proof, let us deal with the case $x_j^{\kappa} \notin G$. In that case, we have $P_E x_j^{\kappa} \neq 0$. In this case, we have
$$
    \sup_{ e \in \mathbb{S}^{d-1}} \langle P_E x_j^{\kappa}, e \rangle  = |P_E x_j^{\kappa} | .
$$
It follows from the information available for  $u_j$ that
$$
 H\left( x_j^{\kappa} , \left|M x_j^{\kappa} +  b + \xi_j + \kappa \frac{P_E x_j^{\kappa}}{|P_E x_j^{\kappa}|}  \right| \right)   F_j \left( M+ \kappa \left( Id + \frac{P_E x_j^{\kappa}}{|P_E x_j^{\kappa}|} \otimes \frac{P_E x_j^{\kappa}}{|P_E x_j^{\kappa}|}\right) \right)  \le f_j (x_j^{\kappa}) ,
$$
where 
$$
 H (x_j^{\kappa}, \xi ) = \left[ \sigma_j \left(|\xi| \right) + a_j(x_j^{\kappa}) \nu_j\left(|\xi | \right) \right]. 
$$
We write $x_j^{\kappa}$ as 
$$
    x_j^{\kappa} = \sum\limits_{i=1}^{d} c_i e_i,
$$
where $\{ e_i, i = 1, \dots, d \}$ are the eigenvectors of $M$. Hence, 
$$
    M x_j^{\kappa} = \sum\limits_{i=1}^{k} \tau_i c_i e_i + \sum\limits_{i=k+1}^{d} \tau_i c_i e_i,
$$
with $\tau_{i} > 0$ for $i = 1, \dots, k.$ Since, the sequences $|M x_j^{\kappa} |$ and $ |b + \xi_j|$ go to zero, we can fix $\tilde{j} \in \mathbb{N}$ such that $|M x_j^{\kappa} | + |b + \xi_j| \le \frac{\kappa}{2}$ for every $j > \tilde{j}$. Hence,  
\begin{eqnarray}
    \left|  Mx_j^{\kappa} +  b + \xi_j + \kappa \frac{P_E x_j^{\kappa}}{|P_E x_j^{\kappa}|} \right| &\ge& \kappa - |Mx_j^{\kappa}  + b + \xi_j| 
 > \frac{\kappa}{2}. \nonumber
\end{eqnarray}
Moreover, 
$$
    \kappa \left( Id + \frac{P_E x_j^{\kappa}}{|P_E x_j^{\kappa}|} \otimes \frac{P_E x_j^{\kappa}}{|P_E x_j^{\kappa}|}  \right) \ge 0
$$
in the sense of matrices.  Therefore, by combining the ellipticity of $F$ with the monotonicity of  $\sigma_j$ and $\nu_j$, we get
$$
    F_{j}(M) \le F_j \left( M + \kappa \left( Id + \frac{P_E x_j^{\kappa}}{|P_E x_j^{\kappa}|} \otimes \frac{P_E x_j^{\kappa}}{|P_E x_j^{\kappa}|}\right) \right) \le \frac{f_j(x_j^{\kappa})}{\sigma_j (\frac{\kappa}{2}) + a_j(x_j^{\kappa}) \nu_j (\frac{\kappa}{2})} \le \frac{f_j(x_j^{\kappa})}{\sigma_j (\frac{\kappa}{2} )}\to 0 ,
$$
as $j \to \infty$ as desired. 
\smallskip

\medskip

\noindent {\it Step 4.} Therefore, we can conclude that $F_{\infty} (M) \le 0$ and $u_{\infty}$ is a supersolution to $F_{\infty} = 0$ in the viscosity sense. In addition, we notice that to prove that $u_{\infty}$ is also a subsolution is analogous to the previous case and we omit the details here. Finally, by applying standard results in the regularity theory of viscosity solutions to elliptic equations, for instance  \cite{Caffarelli1989} and \cite{Trudinger1988}, yield that $u_{\infty}$ belongs to ${\mathcal C}_{loc}^{1,\beta} (B_1)$, for some (universal) $\beta \in (0,1)$. We obtain a contradiction by taking $h = u_{\infty}$ and complete proof. 
\end{proof}

\section{Regularity results}\label{Regularity sct}

In this section, we prove the existence of a sequence of affine functions that approximate normalized viscosity solutions to \eqref{eq_perturbed}. The strategy of the proof follows Caffarelli's iteration scheme, which, adapted to our context, consists of finding sequences of affine functions converging to the linear part of $u$. In order to do that, at each iteration,  we generate a sequence of scaled solutions to 
$$
    \mathfrak{F}_k (x, Du_k + \xi_k, D^2 u_k ) = f_k  \quad \text{in} \quad B_1
$$
where $\xi_k \in \mathbb{R}^d$ is an arbitrary vector  and, at each iteration, the new operator $\mathfrak{F}_k (x, z, M )$ has diffusion agent $F_k$ and a law of degeneracy given by 
$$
    \sigma_k (t) + a_k(x) \nu_k (t),
$$
where $0\le a_k (\cdot) \in \mathcal{C}(B_1)$. From here, the goal is to employ the Proposition \ref{approx lemma} to find an $F$-harmonic function in a small neighborhood of the new scaled solution. A careful inspection of the proof of Proposition \ref{approx lemma} reveals that it suffices for $\sigma_k$ to belong to a family of non-collapsing moduli of continuity for every $k \in \mathbb{N}$. Another examination of our iteration schemes below reveals that $a_k(x) = a_k(r^k x)$ and that ${\nu}_k$ will be a modulus of continuity. Therefore, in order to apply Proposition \ref{approx lemma} we only need to ensure that $\sigma_k$ is, in fact, a non-collapsing modulus of continuity for all $k \in \mathbb{N}$. 

In the sequel, we follow the construction employed in \cite{APPT} to get a shored-up sequence of moduli of continuity. As a consequence, the (double) degeneracy law will satisfy the non-collapsing property as in Proposition \ref{noncollapsing for inhomog}.

\subsection{The construction of the shored-up sequence} 

Let $\gamma(t) = t\nu(t)$ and $\varpi (t) = \gamma^{-1} (t)$. In what follows, let $L>0$ and $\beta \in (0,1)$ as in Lemma \ref{Holder-continuity-lemma},  we first consider the case where $t^{\beta} = o(\varpi(t))$. In this case, we choose $0 < r< 1/2$ so small that 
$$
    2Lr^{\beta}: = \mu_1 > r.
$$
On the contrary, if $\varpi(t) = O(t^{\beta})$, we fix $0 < \alpha < \beta $ and $0 < r < 1/2$ so that 
$$
    2 L r^{\beta} = r^{\alpha}=: \mu_1 > r.
$$
Next, we define 
$$
    0 < \theta = \frac{r}{\mu_1} < 1.
$$
Since, under the Assumption \eqref{Dini cond for nu},  $\nu^{-1}$ is  Dini continuous, we have that 
$$
    (a_k)_{k \in \mathbb{N}}: = \left( \nu^{-1} ( \theta^k ) \right)_{ k \in \mathbb{N}}
$$ is a summable sequence. That is, $(a_k)_{k \in \mathbb{N}} \in \ell_1$. Hence, we resort to the Lemma \ref{existence of noncollapsing}. For $0 < \delta < 1/4$, we define $\varepsilon \in (0, 1)$ as $\varepsilon: = ( 1 + \delta)^{-1}$. Applying Lemma \ref{existence of noncollapsing} yields a positive sequence $(c_k)_{k \in \mathbb{N}} \in c_0$ such that 
\begin{equation}\label{estimate_c0}
   \frac{7}{10} \sum\limits_{k = 1}^{\infty} \nu^{-1} ( \theta^k ) \le \sum\limits_{k = 1}^{\infty} \frac{\nu^{-1} ( \theta^k )}{c_k} \le \sum\limits_{k = 1}^{\infty} \nu^{-1} ( \theta^k ). 
\end{equation}
Finally, we generate two sequences of moduli of continuity $(\sigma_k (\cdot) )_{k \in \mathbb{N}}$ and $(\nu_k (\cdot) )_{k \in \mathbb{N}}$ given by 
$$
    \sigma_0(t) = \sigma(t) \quad \text{and} \quad \nu_0(t) = \nu(t),
$$
$$
    \sigma_1(t) = \frac{\mu_1}{r}\sigma(\mu_1 t) \quad \text{and} \quad \nu_1(t) = \frac{\mu_1}{r}\nu( \mu_1 t),
$$
\[
    \sigma_2(t) = \frac{\mu_1\mu_2}{r^2}\sigma(\mu_1 \mu_2 t) \quad \text{and} \quad \nu_2(t) = \frac{\mu_1 \mu_2}{r^2}\nu( \mu_1\mu_2 t),
\]
$$
    \vdots
$$

$$
    \sigma_k(t) = \frac{\prod_{j=1}^k \mu_j}{r^k}\sigma\left( \prod_{j=1}^k \mu_j t \right)  \quad \text{and} \quad \nu_k(t) = \frac{\prod_{j=1}^k \mu_j}{r^k}\nu\left( \prod_{j=1}^k \mu_j  t\right),
$$
where $\mu_1 > r$ has been selected. By applying the same recursive algorithm as in \cite{APPT}, we obtain $(\mu_k)_{k \in \mathbb{N}}$ such that $r<\mu_1 \le \mu_2 \le \ldots\le \mu_k$ and we select the value of $\mu_{k+1}$ as follows: if 
$$
    \frac{\prod_{j=1}^{k+1 }\mu_j}{r^{k+1 }} \nu \left( \prod_{j=1}^{k+1 } \mu_j c_{k +1}\right) \ge 1,
$$
then, we set $\mu_ {k+1} = \mu_{k}$. Otherwise, $\mu_{k} < \mu_{k+1}$ is chosen so that 
$$
    \frac{\prod_{j=1}^{k+1 }\mu_j}{r^{k+1 }} \nu \left( \prod_{j=1}^{k+1 } \mu_j c_{k +1}\right) = 1. 
$$
where $c_{k+1}$ is the $(k+1)$-th element of the sequence $(c_k)_{k \in \mathbb{N}} \in c_0$ that satisfies \eqref{estimate_c0}. Note that the Assumption A\ref{assump_moduli} ensures that 
$$
    \frac{\prod_{j=1}^{k+1 }\mu_j}{r^{k+1 }} \sigma \left( \prod_{j=1}^{k+1 } \mu_j c_{k +1}\right) \ge 1.
$$
We are now ready to construct a sequence of approximation hyperplanes whose difference with $u$ grows in a controlled fashion. In what follows, we consider
\begin{equation}\label{eq_perturbed}
\left[\sigma(|Du + \xi|) + a(x)\nu(|Du + \xi|) \right]F(D^2 u) = f(x) \quad \text{in} \quad B_1,
\end{equation}
where $\xi$ is an arbitrary vector in $\mathbb{R}^d$. Observe that, if $\xi = 0$, we recover the equation \eqref{eq_main}.

\begin{Proposition}\label{step1-induction}
Let $u \in \mathcal{C}(B_1)$ be a normalized viscosity solution to \eqref{eq_perturbed}. Assume A\ref{A1}-A\ref{assump_f} and \eqref{normalization_moduli} are in force. There exists $\varepsilon > 0$ and $ 0 < r < 1/2$ such that, if $\| f\|_{L^{\infty}(B_1)} \le \varepsilon$,  then we can find an affine function $\ell_0(x) = a_0 + \vec{b}_0 \cdot x$, with universally bounded constants $|a_0| + | \vec{b}_0| \le C(n, \lambda, \Lambda)$ such that 
    \begin{equation}
        \sup\limits_{ x \in B_{r}} |u(x) - \ell_0 (x) | \le   \mu_1  r.
    \end{equation} 
\end{Proposition}

\begin{proof}
    For $\delta > 0$ to be chosen, we can apply Lemma \ref{approx lemma} to find a function $h \in {\mathcal C}_{loc}^{1,\beta} (B_1)$ such that 
    $$
        \sup\limits_{B_{9/10}} |u(x) - h(x) |  < \delta.
    $$
    From the regularity theory available for $h$, we have 
    $$
        \sup\limits_{x \in B_r} | h(x) - h(0) - Dh(0) \cdot x | \le L r^{1+\beta}
    $$
    for some universal constant $L > 0$ and for every $0 < r < 1/2. $ By choosing $a_0 := h(0)$ and $\vec{b}_0 := Dh(0)$, then the triangular inequality yields
    \begin{eqnarray}
        \sup\limits_{ x \in B_{r}} | u(x) - (a_0 + \vec{b}_0\cdot x)| &<& \delta + Lr^{1 + \beta} \nonumber \\
        &=& \delta + \frac{1}{2} \mu_1 r.
    \end{eqnarray}
   Finally, by choosing $\delta = \frac{1}{2} \mu_1 r$, we fix the value of $\varepsilon > 0$, through Lemma \ref{approx lemma}, and the proof is completed. 
\end{proof}

In the sequel, we iterate the above proposition to find a sequence of approximating hyperplanes whose difference with a solution $u$ is controlled at discrete scales. 

\begin{Proposition}
Let $u \in \mathcal{C}(B_1)$ be a normalized viscosity solution to \eqref{eq_perturbed}. Assume A\ref{A1}-A\ref{assump_f} and \eqref{normalization_moduli} hold true. Then, we can find a sequence of affine functions $(\ell_n)_{n \in \mathbb{N}}$ of the form

\[
\ell_n(x):= A_n + B_n \cdot x
\]
satisfying 
\begin{equation*}
 \sup_{x \in B_{r^n}} |u(x) - \ell_n(x)| \le \left( \prod_{i = 1}^n \mu_i \right) r^n,   
\end{equation*}

\begin{equation*}
 |A_{n+1} - A_n| \le C \left( \prod_{i = 1}^n \mu_i \right) r^n,   
\end{equation*} 
and 
\begin{equation*}
 |B_{n+1} - B_n| \le C \prod_{i = 1}^n \mu_i,   
\end{equation*} 
for every $n \in \mathbb{N}$.
\end{Proposition}
\begin{proof}
The proof follows the reasoning of  \cite[Proposition 7]{APPT}. For completeness, we carry out the details here. For easy presentation, we split the proof into two steps.

\medskip
\noindent {\it Step 1.} We start  by considering  the auxiliary function
\[
u_1(x):= \frac{u(rx) - \ell_0(rx)}{\mu_1 r}
\]
where  $\mu_1$ and $\ell_0$ as in Proposition \ref{step1-induction}. We have $u_1$ solves 
\[
\left[\sigma_1\left( \left|Du_1 + \frac{1}{\mu_1} D \ell \right|\right) + a_1(x) \nu_1\left( \left|Du_1 + \frac{1}{\mu_1} D \ell \right|\right)\right]F_1(D^2 u_1) = f_1(x) \quad \text{in} \quad B_1, 
\]
where
$$
    F_1(M):= \frac{r}{\mu_1} F\left( \frac{\mu_1}{r} M \right),
$$

\[
\sigma_1(t):= \frac{\mu_1}{r} \sigma(\mu_1 t) \quad \text{and} \quad \nu_1(t):= \frac{\mu_1}{r} \nu(\mu_1 t),
\]
and 
\[
a_1(x) := a(r x) \quad \text{and} \quad f_1(x):= f(rx).
\]
Since $\nu_1 (1) \ge 1$ by construction, then $u_1$ is under the scope of Proposition \ref{step1-induction}. Thus, we can find an affine function $\ell_1$ with universal bounds, such that 
$$
    \sup\limits_{ x \in B_r} |u_1(x) - \ell_1 (x) | \le r \mu_1.
$$
Next,  let us define $u_2$ as
$$
    u_2 (x) = \frac{u_1(rx) - \ell_1(rx)}{\mu_2 r},
$$
for $0 < r < \mu_1 \le \mu_2$ chosen earlier. We readily check that $u_2$ solves
\[
\left[\sigma_2\left( \left|Du_2 + \frac{1}{\mu_2} D \ell_1 \right|\right) + a_2(x) \nu_2\left( \left|Du_2 + \frac{1}{\mu_2} D \ell_1 \right|\right)\right]F_2(D^2 u_2) = f_2(x) \quad \text{in} \quad B_1, 
\]
where
$$
    F_2(M):= \frac{r}{\mu_2} F_1\left( \frac{\mu_2}{r} M \right),
$$

\[
\sigma_2(t):= \frac{\mu_1 \mu_2}{r^2} \sigma(\mu_1 \mu_2 t) \quad \text{and} \quad \nu_2(t):= \frac{\mu_1 \mu_2}{r} \nu( \mu_1 \mu_2 t),
\]
and 
\[
a_2(x) := a_1(r^2 x) \quad \text{and} \quad f_2(x):= f_1(r^2x).
\]
Since $u_2$ satisfies the assumptions of the Proposition \ref{step1-induction}, which implies that there exists an affine function $\ell_2$, with universal bounds such that
\[
\sup_{x \in B_r} |u_2(x) - \ell_2(x)| \le r \mu_1.
\]
Iterating inductively the above reasoning, we obtain that 
\[
u_{k+1}(x):= \frac{u_k(rx) - \ell_k(rx)}{\mu_{k+1} r}
\]
solves an equation with a degeneracy law given by 
$$
    H_k(x, t) := \sigma_{k+1} (t) + a_{k+1}(x) \nu_{k+1} (t)
$$
where
$$
    \sigma_{k+1} (t) =\frac{\prod_{i=1}^{k+1} \mu_i }{r^{k+1}} \sigma \left( \prod_{i=1}^{k+1} \mu_i t \right) , \quad \nu_{k+1} (t) =\frac{\prod_{i=1}^{k+1} \mu_i }{r^{k+1}} \nu \left( \prod_{i=1}^{k+1} \mu_i t \right) 
$$
and 
$$
    a_{k+1}(x) = a(r^{k+1} x) .
$$
We recall that $\mu_{k+1} \ge \mu_{k}$ is chosen in such way that either $\mu_{k+1} = \mu_{k}$, if
\begin{equation*}
 \frac{\prod_{i=1}^{k+1} \mu_i }{r^{k+1}} \nu \left( \prod_{i=1}^{k+1} \mu_i c_{k+1} \right) \ge 1,
\end{equation*}
or $\mu_{k+1} > \mu_{k}$, if
\begin{equation*}
 \frac{\prod_{i=1}^{k+1} \mu_i }{r^{k+1}} \nu \left( \prod_{i=1}^{k+1} \mu_i c_{k+1} \right) = 1.
\end{equation*}
As before, we apply the Proposition \ref{step1-induction} to ensure that there exists an affine function $\ell_{k+1}$ satisfying
\[
\sup_{x \in B_r} |u_{k+1}(x) - \ell_{k+1}(x)| \le \mu_1 r.
\]

\medskip

\noindent {\it Step 2.} Scaling back to the solution $u$, we have that
\[
\sup_{x \in B_{r^k}}|u(x) - \ell_k(x)| \le \left( \prod^k_{i=1} \mu_i\right) r^k,
\]
where $\ell_k$ is given by
\[
\ell_k(x):= \ell_0(x) + \sum_{i=1}^{k-1} \ell_i(r^{-i} x) \left( \prod^i_{j=1} \mu_j\right) r^i.
\]
In addition, we can write $\ell_k(x)$ as $ A_k + B_k \cdot x$ and obtain the following estimates

\[
|A_{k+1} - A_k | \le C \left( \prod^{k}_{i=1} \mu_i\right) r^{k},
\]
and 
\[
|B_{k+1} - B_k | \le C \left( \prod^{k}_{i=1} \mu_i\right),
\]
which finishes the proof.
\end{proof}

\bigskip

To finalize the proof of our differentiability result, we need to ensure the convergence of the approximating hyperplanes. This rate of convergence will result in a modulus of continuity $\omega$ associated with the product $\prod_{i=1}^k \mu_i$.

\begin{proof}[Proof of Theorem \ref{main_theorem}]
The proof follows along the lines of \cite{APPT}. For this reason, we only highlight the main steps. The construction of the sequence $(\mu_k)_{k \in \mathbb{N}}$ plays a fundamental role in the proof. We recall that either $\mu_{k_0} = \mu_{k_0 +1}= \mu_{k_0 +2}= \ldots$ for some $k_0\ge 2$, or $\mu_k <\mu_{k+1}$, if
\begin{equation*}
 \frac{\prod_{i=1}^{k+1} \mu_i }{r^{k+1}} \nu \left( \prod_{i=1}^{k+1} \mu_i c_{k+1} \right) = 1.
\end{equation*}
The first case fits within the classical framework, wherein we establish that the solutions are locally of class $\mathcal{C}^{1, \tilde{\beta}}$ for some $\tilde{\beta} \in (0, \beta)$. Next, we analyze the latter case. Define
$$
    (\tau_k)_{k \in \mathbb{N}} = \left( \prod_{i=1}^k \mu_i \right)_{k \in \mathbb{N}}. 
$$
Since $\nu^{-1}$ is Dini continuous, we can conclude as in \cite{APPT} and as in \cite[Proof of Theorem 2]{PS} that this sequence is summable and bounded by $\sum_{i=1}^{\infty} \nu^{-1}(\theta^i)$,  where $\theta^i:= r^i/ \prod_{j=1}^i\mu_j$. Moreover,  $(A_n)_{n \in \mathbb{N}}$, $(B_n)_{n \in \mathbb{N}}$ are Cauchy sequences. Therefore, there exist real numbers $A_{\infty}$, $B_{\infty}$ such that 
\[
A_n \rightarrow A_{\infty} \quad \text{and}  \quad 
B_n \rightarrow B_{\infty}.
\]
Now, define  $\ell_{\infty} (x) = A_{\infty} + B_{\infty} \cdot x$, we have 
$$
    |A_{\infty} - A_n| \le C \sum_{i=n}^{\infty} \tau_i r^n\quad \text{and} \quad | B_{\infty} - B_n| \le C \sum_{i=n}^{\infty}\tau_i r^n.
$$
For any $0 < \rho \ll 1$, let $n \in \mathbb{N}$ such that 
$$
    r^{n+1} < \rho \le r^n,
$$
then, we can estimate
\begin{eqnarray}
    \sup\limits_{x \in B_{\rho}} |u(x) -  \ell_{\infty}(x)| &\le&   \sup\limits_{x \in B_{r^n}} |u(x) - \ell_n (x) | +  \sup\limits_{x \in B_{r^n}} |\ell_n(x) - \ell_{\infty} (x) |  \nonumber \\
    &\le& C\tau_n r^n + C\left( \sum_{i=n}^{\infty} \tau_i \right) r^n \nonumber \\
    &\le&  C\left( \sum_{i=n}^{\infty} \tau_i \right) \rho .
\end{eqnarray}
Next, by setting 
$$
    \omega (t) = C\left( \sum_{i= \lfloor \ln t^{-1}\rfloor }^{\infty} \tau_i \right) 
$$
where $\lfloor s \rfloor$ is  the biggest integer less or equal than $s$, we see that $\omega(t)$ is indeed a modulus of continuity, since $(\tau_k)_{k \in \mathbb{N}}$ is summable sequence. Finally, we have 
$$
    \sup\limits_{x \in B_{\rho}} |u(x) - u(0) - D u(0) \cdot x | \le \omega (\rho) \rho
$$
which completes the proof of Theorem \ref{main_theorem}.
\end{proof} 
\begin{Remark}
Notice that Theorem \ref{main_theorem} naturally extends for equations with multi-phase degeneracy of the form 
\begin{equation}
\left[\sigma(|Du|) + \sum_{i=1}^k a_i(x)\sigma_i(|Du|) \right] F(D^2 u) = f(x) \quad \text{in} \quad B_1.
\end{equation}
As long as we consider the following decay: 
$$
    \sigma (t) = \sigma_0 (t) \ge \sigma_1(t) \ge \cdots \ge \sigma_k (t), \, \, \text{for all}\, \, t \in [0,1]
$$
and $\sigma_k^{-1} (t)$ satisfy the Dini continuity condition and $a_i(\cdot)$ are nonnegative and continuous functions for all $i= 1, \ldots, k$.

\end{Remark}

\begin{Remark}
Assume that the moduli of continuity $\sigma $ and $\nu$ in the equation \eqref{eq_main} satisfy the following condition:
\[
\lim_{t \rightarrow 0 } \frac{\sigma(t)}{t^p}\le C_1 \quad \text{and} \quad \lim_{t \rightarrow 0 } \frac{\nu(t)}{t^q}\le C_2,
\]
where $p < q$ and $C_1, C_2$ nonnegative constants. Under these assumptions, we believe it is possible to prove  H\"older continuity of the gradient of solutions. 
\end{Remark}

\bigskip
\noindent {\bf Acknowledgements:} The authors extend their sincere gratitude to the Department of Mathematics at Iowa State University for its support and hospitality. P\^edra D. S. Andrade is partially supported by the Portuguese government through FCT - Funda\c c\~ao para a Ci\^encia e a Tecnologia, I.P., under the project UIDP/00208/2020 (DOI: 10.54499/UIDP/00208/2020) and the Austrian Science Fund (FWF) project 10.55776/P36295.

\bibliographystyle{plain}
\bibliography{biblio}

\bigskip

%\vspace{1cm}
\noindent\textsc{P\^edra D. S. Andrade}\\
Department of Mathematics\\
Paris Lodron Universität Salzburg\\
5020 Salzburg, Austria\\
\noindent\texttt{pedra.andrade@plus.ac.at}

\vspace{1cm}

\noindent\textsc{Thialita M. Nascimento}\\
Department of Mathematics\\
Iowa State University\\
 50011 Ames, IA, United States\\
\noindent\texttt{thnasc@iastate.edu}

%\vspace{.15in}
%
%\noindent\textsc{Makson S. Santos}\\
%Instituto Superior T\'ecnico \\
%Universidade de Lisboa - ULisboa\\
%1049-001, Av. Rovisco Pais 1, Lisboa, Portugal\\
%\noindent\texttt{makson.santos@tecnico.ulisboa.pt}
%
%\vspace{.15in}
%
%\noindent\textsc{Hugo Tavares}\\
%Instituto Superior T\'ecnico \\
%Universidade de Lisboa - ULisboa\\
%1049-001, Av. Rovisco Pais 1, Lisboa, Portugal\\
%\noindent\texttt{hugo.n.tavares@tecnico.ulisboa.pt}

%\Addresses

\end{document}